\documentclass[12pt,reqno]{amsart}
\usepackage{amsmath,amstext,amssymb,amscd}
\usepackage{verbatim}
\usepackage{enumerate}
\usepackage{mathrsfs}
\usepackage[dvipsnames,usenames]{xcolor}

\usepackage{amsthm}

\newtheorem{theorem}{Theorem}[section]

\newtheorem{proposition}[theorem]{Proposition}

\theoremstyle{definition}

\theoremstyle{remark}
\newtheorem{remark}[theorem]{Remark}

\numberwithin{equation}{section}

\begin{document}

\newcommand{\sgn}{\operatorname{sgn}}

\def\a{\alpha}
\def\b{\beta}
\def\d{\delta}
\def\g{\gamma}
\def\l{\lambda}
\def\o{\omega}
\def\s{\sigma}
\def\t{\tau}
\def\th{\theta}
\def\r{\rho}
\def\D{\Delta}
\def\G{\Gamma}
\def\O{\Omega}
\def\e{\varepsilon}
\def\p{\phi}
\def\P{\Phi}
\def\S{\Psi}
\def\E{\eta}
\def\m{\mu}
\def\grad{\nabla}
\def\bar{\overline}
\newcommand{\reals}{\mathbb{R}}
\newcommand{\naturals}{\mathbb{N}}
\newcommand{\ints}{\mathbb{Z}}
\newcommand{\complex}{\mathbb{C}}
\newcommand{\rationals}{\mathbb{Q}}
\newcommand{\innerprod}[1]{\left\langle#1\right\rangle}
\newcommand{\norm}[1]{\left\|#1\right\|}
\newcommand{\abs}[1]{\left|#1\right|}

\author[P. G\'erard]{Patrick G\'erard}
\address{Universit\'e Paris-Sud XI, Laboratoire de Math\'ematiques d'Orsay, CNRS, UMR 8628,
et Institut Universitaire de France}
\email{Patrick.Gerard@math.u-psud.fr}

\author[Y. Guo]{Yanqiu Guo}
\address{Department of Computer Science and Applied Mathematics \\  Weizmann Institute of Science\\
Rehovot 76100, Israel} \email{yanqiu.guo@weizmann.ac.il}

\author[E. S. Titi]{Edriss S. Titi}
\address{Department of Mathematics and Department of Mechanical and Aerospace Engineering\\
University of California, Irvine, California 92697-3875, USA and Department of Computer
Science and Applied Mathematics, Weizmann Institute of Science, Rehovot 76100, Israel} \email{etiti@math.uci.edu and edriss.titi@weizmann.ac.il}

\title[Analyticity of solutions to the Cubic Szeg\H{o} Equation]
{On the radius of analyticity of solutions\\ to the cubic Szeg\H{o} equation}
\date{Revised: August 6, 2013}
\keywords{Cubic Szeg\H{o} equation, Gevrey class regularity, analytic solutions, Hankel operators}
\subjclass[2010]{35B10, 35B65, 47B35}

\maketitle

\begin{abstract}
This paper is concerned with the cubic Szeg\H{o} equation
$$i\partial_t u=\Pi(|u|^2 u),$$
defined on the $L^2$ Hardy space on the one-dimensional torus $\mathbb T$, where $\Pi: L^2(\mathbb T)\rightarrow L^2_+(\mathbb T)$ is the Szeg\H{o} projector onto the non-negative frequencies.
For analytic initial data, it is shown that the solution remains spatial analytic for all time $t\in (-\infty,\infty)$.
In addition, we find a lower bound for the radius of analyticity of the solution. Our method involves energy-like estimates of the special Gevrey class of analytic functions based on the $\ell^1$ norm of Fourier transforms (the Wiener algebra).
\end{abstract}

\section {Introduction }\label{S1}
In studying the nonlinear Schr\"odinger equation
\begin{align*}
i\partial_t u+\Delta u=\pm|u|^2 u, \;\;\; (t,x)\in \reals \times M,
\end{align*}
Burq, G\'erard and Tzvetkov \cite{Burq-Gerard-Tzvetkov-05} observed that dispersion properties are strongly influenced by the geometry of the underlying manifold $M$. In \cite{Gerard-10}, G\'erard and Grellier mentioned, if there exists a smooth local in time flow map on the Sobolev space $H^s(M)$, then the following Strichartz-type estimate must hold:
\begin{align} \label{Strichartz}
\norm{e^{it\Delta}f}_{L^4([0,1]\times M)}\lessapprox \norm{f}_{H^{s/2}(M)}.
\end{align}
It is shown in \cite{Burq-Gerard-Tzvetkov-02, Burq-Gerard-Tzvetkov-05} that, on the two-dimensional sphere, the infimum of the number $s$ such that (\ref{Strichartz}) holds is $\frac{1}{4}$; however, if $M=\reals^2$, the inequality (\ref{Strichartz}) is valid for $s=0$. As pointed out in \cite{Gerard-10}, this can be interpreted as a lack of dispersion properties for the spherical geometry. Taking this idea further, it is remarked in \cite{Gerard-10} that dispersion disappears completely when $M$ is a sub-Riemannian manifold (for instance, the Heisenberg group).

As a toy model to study non-dispersive Hamiltonian equation, G\'erard and Grellier \cite{Gerard-10} introduced the \emph{cubic Szeg\H{o} equation} :
\begin{align} \label{Szego}
i\partial_t u=\Pi(|u|^2 u), \;\;\;(t,\theta)\in \reals \times \mathbb T,
\end{align}
on $L^2_+(\mathbb T)$, where $\mathbb T=\reals  /2\pi \ints$ is the one-dimensional torus, which is identical to the unit circle in the complex plane. Notice that $L^2_+(\mathbb T)$ is the $L^2$ Hardy space which is defined by
\begin{align*}
L^2_+(\mathbb T)=\Big\{u=\sum_{k\in \ints}\hat u(k) e^{ik\theta}\in L^2(\mathbb T): \hat u(k)=0 \text{\;\;for all\;\;} k<0\Big\}.
\end{align*}
Furthermore, in (\ref{Szego}), the operator $\Pi: L^2(\mathbb T)\rightarrow L^2_+(\mathbb T)$ is the Szeg\H{o} projector onto the non-negative frequencies, i.e.,
\begin{align*}
\Pi\left(\sum_{k\in \ints}v_k e^{ik\theta}\right)=\sum_{k\geq 0}v_k e^{ik\theta}.
\end{align*}
We mention the following existence result, proved in \cite{Gerard-10}.
\begin{theorem} \cite{Gerard-10} \label{thm-Gerard}
Given $u_0\in H^{s}_+(\mathbb T)$, for some $s\geq \frac{1}{2}$, then the cubic Szeg\H{o} equation (\ref{Szego}) has a unique solution  $u\in C(\reals,H^s_+(\mathbb T))$.
\end{theorem}
Moreover, it has been shown in \cite{Gerard-10} that the Szeg\H{o} equation (\ref{Szego}) is completely integrable in the sense of admitting a Lax pair structure, and as a consequence, it possesses an infinite number of conservation laws.

Replacing the Fourier series by the Fourier transform, one can analogously define the Szeg\H{o} equation on $$L^2_+(\reals)=\{\phi\in L^2(\reals): \text{supp\;} \hat{\phi} \subset [0,\infty) \}.$$
In \cite{Pocovnicu-11}, Pocovnicu constructed explicit spatially real analytic solutions for the cubic Szeg\H{o} equation defined on $L^2_+(\reals)$. For the initial datum
$u_0=\frac{2}{x+i}-\frac{4}{x+2i}$, it was discovered that one of the poles of the explicit real analytic solution $u(t,x)$ approaches the real line, as $|t|\rightarrow \infty$; more precisely, the imaginary part of a pole decreases in the speed $O(\frac{1}{t^2})$.
Thus, the radius of analyticity of $u(t,x)$ shrinks algebraically to zero, as $|t|\rightarrow \infty$. This phenomenon gives rise to the following questions: for analytic initial data, does the solution remain spatial analytic for all time? If so, can one estimate, from below,  the radius of analyticity? In this manuscript, we attempt to answer these questions by employing the technique of the so--called Gevrey class of analytic functions.

The Gevrey classes of real analytic functions are characterized by an exponential decay of their Fourier coefficients. If we set $A:=\sqrt{I-\Delta}$, they are defined by $\mathcal D(A^s e^{\sigma A})$, which consist of all $L^2$ functions $u$ such that $\norm{A^s e^{\sigma A}u}_{L^2(\mathbb T)}$ is finite, where $s\geq 0$, $\sigma>0$ (see e.g. \cite{Ferrari-Titi-98, Foias-Temam-89, Levermore-Oliver-97}). Note, if $\sigma=0$, then $\mathcal D(A^s e^{\sigma A})=\mathcal D(A^s)\cong H^s(\mathbb T)$. However, if $\sigma>0$, then $\mathcal D(A^s e^{\sigma A})$ is the set of real analytic functions with the radius of analyticity bounded below by $\sigma$.
Also notice, $\mathcal D(A^s e^{\sigma A})$ is a Banach algebra provided $s>\frac{1}{2}$ for 1D (see \cite{Ferrari-Titi-98}).

The so--called method of Gevrey estimates has been extensively used in literature to establish regularity results for nonlinear evolution equations. It was first introduced for the periodic Navier-Stokes equations in \cite{Foias-Temam-89}, and studied later in the whole space in \cite{Oliver-Titi-00}, moreover, it was extended to nonlinear analytic parabolic PDE's in \cite{Ferrari-Titi-98}, and for Euler equations in \cite{Kukavica-Vicol-09, Larios-Titi-10, Levermore-Oliver-97} (see also references therein). Recently, this method was  also applied to establish analytic solutions for  quasilinear wave equations \cite{Guo-Titi-12}.

In this paper, we employ a special such class based on the space $W$ of functions with summable Fourier series. For a given function $u\in L^1(\mathbb T)$, $u=\sum_{k\in \ints} \hat u(k) e^{ik\theta}$,
$\theta\in \mathbb T$, then the Wiener norm of $u$ is given by
\begin{align} \label{ell}
\norm {u}_W=\norm{\hat u}_{\ell^1}=\sum_{k\in \ints} |\hat u(k)|.
\end{align}
Notice that $W$  is a Banach algebra (Wiener algebra).

Based on the Wiener algebra, the following special Gevrey norm is defined in \cite{Oliver-Titi-01}:
\begin{align} \label{G-norm}
\norm{u}_{G_{\s}(W)}=\sum_{k\in \ints}e^{\s|k|}|\hat u(k)|, \;\;\;\sigma\geq 0.
\end{align}
If $u\in L^1(\mathbb T)$ is such that $\norm{u}_{G_{\s}(W)}<\infty$, then we write $u\in G_{\s}(W)$.

It is known that the Gevrey class $G_{\s}(W)$ is a Banach algebra \cite{Oliver-Titi-01}, and  it characterizes the real analytic functions if $\s>0$. In particular, a function $u\in C^{\infty}(\mathbb T)$ is real analytic with uniform radius of analyticity $\rho$, if and only if, $u\in G_{\s}(W)$, for every $0<\s<\rho$.
\bigskip

Now, we state the main result of this paper.
\begin{theorem} \label{main-thm}
Assume $u_0\in L^2_+(\mathbb T)\cap  G_{\sigma}(W)$, for some $\sigma>0$. Then the unique solution $u(t)$ of (\ref{Szego}) provided by Theorem \ref{thm-Gerard} satisfies $u(t)\in G_{\t(t)}(W)$, for all $t\in \reals$,
where $\t(t)= \sigma e^{-\l |t|}$, with some $\l>0$ depending on $u_0$.  More precisely, there exists $C_0>0$, specified in (\ref{C0}) below, such that $\norm{u(t)}_{G_{\t(t)}(W)}\leq C_0$, for all $t\in \reals$.
\end{theorem}

Essentially, Theorem \ref{main-thm} shows the persistency of the spatial analyticity of the solution $u(t)$ for all time $t\in (-\infty,\infty)$ provided the initial datum is analytic. Recall that $\t(t)$ is a lower bound of the radius of spatial analyticity of $u(t)$. Thus, it implies that the radius of analyticity of $u(t)$ cannot shrink faster than exponentially, as $|t|\rightarrow \infty.$

\begin{remark}
The precise definition of $\l $ in Theorem \ref{main-thm} is given in (\ref{def-tau}) below. In fact, as shown in Remark \ref{improvement},
one can prove that the radius $\rho (t)$ of real analyticity of $u(t)$ satisfies, for every $s>1$,
$$\limsup_{t\rightarrow \infty} \left| \frac{\log \rho (t)}{t} \right| \le K_s \norm {u_0}_{H^s}^2\ ,$$
which is independent of the $G_{\s }(W)$ norm of $u_0$. The optimality of such an estimate is not known. However, let us mention the following two recent results in \cite{Gerard-13}. Firstly, if $u_0$ is a rational function of $e^{i\theta }$ with no poles in the closed unit disc, then so is $u(t)$,
and $\rho (t)$ remains bounded from below by some positive constant for all time. Secondly, this bound is by no means uniform. Indeed, starting with
$$u_0=e^{i\theta }+\e \ ,\ \e >0\ ,$$
one can show that
$$\rho \left (\frac \pi \e \right )=O(\e ^2)\ .$$
This phenomenon is to be compared to the one displayed by Kuksin in \cite{Kuksin} for NLS on the torus with small  dispersion coefficient.

Finally, let us mention a recent work by Haiyan Xu \cite{Xu-13}, who found a Hamiltonian perturbation of the cubic Szeg\H{o} equation which admits solutions
with exponentially shrinking radius of analyticity. Moreover, one can check that the method of Theorem \ref{main-thm} applies as well to this perturbation, so that the above result is optimal in the case of this equation.

\end{remark}

By investigating the steady state of the cubic nonlinear Schr\"odinger equation, it is demonstrated in \cite{Oliver-Titi-01} that, by employing the Gevrey class $G_{\s}(W)$, one can obtain a more accurate estimate of the lower bound of the radius of analyticity of solutions to differential equations, compared to the estimate derived from using the regular Gevrey classes $\mathcal D(A^s e^{\s A})$ (see also the discussion in \cite{Guo-Titi-12}).
Such observation is verified again in this paper, since we find that, in studying the cubic Szeg\H{o} equation, the Gevrey class method, based on $G_{\s}(W)$, provides an estimate of the lower bound of the analyticity radius of the solution, which has a substantially slower shrinking rate, than the estimate obtained from using the  classes $\mathcal D(A^s e^{\s A})$. One may refer to Remark \ref{rmk-compare} for this comparison.

Throughout, we study the cubic Szeg\H{o} equation defined on the torus $\mathbb T$. However, by using Fourier transforms instead of Fourier series, our techniques are also applicable to the same equation defined on the real line, and similar  regularity results and estimates can be obtained as well (see also ideas from \cite{Oliver-Titi-00}).

Moreover, Theorem \ref{main-thm} is also valid under the framework of general Gevrey classes, i.e.,
intermediate spaces between the space of $C^{\infty}$ functions and real analytic functions. Indeed, if we define Gevrey classes $G_{\sigma}^{\gamma}(W)$ based on the norm
\begin{align*}
\norm{u}_{G_{\sigma}^{\gamma}(W)}=\sum_{k\in \mathbb Z} e^{\sigma |k|^{\gamma}}|\hat u (k)|, \;\;\gamma\in (0,1],
\end{align*}
then, $G_{\sigma}^{\gamma}(W)$ are Banach algebras, due to the elementary inequality
$e^{\sigma (k+j)^{\gamma}}\leq e^{\sigma k^{\gamma}} e^{\sigma j^{\gamma}}$, for $\gamma\in (0,1]$.
Thus, the proof of Theorem \ref{main-thm} works equally for $G_{\sigma}^{\gamma}(W)$, where $\gamma\in (0,1]$. For the sake of clarity, we demonstrate our technique for $\gamma=1$, i.e., the Gevrey class of real analytic functions.

\bigskip

\section{Proof of the main result}
Before we start the proof of  the main result, the following proposition should be mentioned.
\begin{proposition} \label{prop}
Assume $u_0\in H^s_+(\mathbb T)$, for some $s>1$. Let $u$ be the unique global solution of (\ref{Szego}), furnished by Theorem \ref{thm-Gerard}. Then,
\begin{align} \label{ell-bound}
\norm{u(t)}_{W}\leq C(s)\norm{u_0}_{H^s}, \text{\;\;for all\;\;} t\in \reals.
\end{align}
\end{proposition}
\begin{proof}
The proof can be found in  \cite{Gerard-10}, we recall it here. In  \cite{Gerard-10}, it has been shown that the cubic Szeg\H{o} equation admits a Lax pair $(H_u, B_u)$ ,  where $H_u$ is the Hankel operator of symbol $u$, defined by
\begin{align}\label{hankel}
H_u(h)=\Pi (u\overline h)\ .
\end{align}
Thus the trace norm $Tr(|H_{u(t)}|)$ is a conserved quantity. By Peller's theorem \cite{Pe80}, \cite{P}, $Tr(|H_u|)$ is equivalent to the $B_{1,1}^1$ norm of $u$. In particular, for every $s>1$,
\begin{align} \label{tracebound}
\frac 1 2\norm{u}_{W}\leq  Tr(|H_u|)\leq C_s \norm{u}_{H^s}\ .
\end{align}
Hence
$$\norm{u(t)}_{W}\leq 2 Tr(|H_{u(t)}|)=2 Tr(|H_{u_0}|)\le 2C_s \norm{u_0}_{H^s}\ .$$
The proof is complete. \end{proof}
For the sake of completion, we provide a straightforward proof of (\ref{tracebound})  in the Appendix.
We now start the proof of Theorem \ref{main-thm}.
\begin{proof}
Due to the assumption on the initial datum $u_0$, we know that $u_0$ is real analytic, and hence $u_0\in H^s_+(\mathbb T)$, for every non-negative real number $s$, in particular for $s\geq \frac{1}{2}$. Therefore, the global existence and uniqueness of the solution $u\in C(\reals,H^s_+(\mathbb T))$ are guaranteed by Theorem \ref{thm-Gerard}, for $s\geq \frac{1}{2}$.

Throughout, we focus on the positive time $t\geq 0$. By replacing $t$ by $-t$, the same proof works for the negative time.

We shall implement the Galerkin approximation method. Recall the cubic Szeg\H{o} equation is defined on the Hardy space $L^2_+(\mathbb T)$ with a natural basis $\{e^{ik\theta}\}_{k\geq 0}$.
Denote by $P_N$ the projection onto the span of $\{e^{ik\theta}\}_{0\leq k\leq N}$.
We let
\begin{align} \label{Appro-Sol}
u_N(t)=\sum_{k=0}^{N} \hat u_{N}(t,k)e^{ik\theta}
\end{align}
be the solution of the Galerkin system:
\begin{align} \label{Galerkin}
i\partial_t u_N=P_N \left(|u_N|^2 u_N\right),
\end{align}
with the initial condition $u_N(0)=P_N u_0$. We see that (\ref{Galerkin}) is an $N$-dimensional system of ODE with the conservation law
$$\norm {u_N}_{L^2}^2=\sum _{k=0}^N \vert \hat u_N(t;k)\vert ^2\ ,$$
and thus it has a unique solution $u_N\in C^\infty (\reals )$ on $\reals $.

Arguing exactly as in section 2 of \cite{Gerard-10}, we observe that
$$\sum _{k=0}^N k\vert \hat u_N (k)\vert ^2 $$
is a conservation law, hence $\norm {u_N(t)}_{H^{1/2}} $ is conserved, consequently, for every $s\ge \frac 12$ and every $T>0$,
\begin{align*}
\sup _N\sup _{t\in [0,T]}\norm{u_N(t)}_{H^s}<\infty \ .
\end{align*}
By using the equation (\ref{Galerkin}), one concludes that the same estimate holds for the time derivative $u_N'(t)$. Now, let us fix an arbitrary $T>0$. Since, moreover, the injection of $H^{s+\e }$ into $H^s$ is compact, we conclude from Ascoli's theorem that, up to a subsequence, $u_N(t)$ converge to some $\tilde u(t)$ in every $H^s$, uniformly for $t\in [0,T]$.
Then, it is straightforward to check, by letting $N\rightarrow \infty$, that $\tilde u$ is a solution of the cubic Szeg\H{o} equation (\ref{Szego}) on $[0,T]$ with the initial datum $u_0$. Since $u$ is the unique global solution furnished by Theorem {\ref{thm-Gerard}}, one must have $u=\tilde u$ on $[0,T]$. Since $H^s$ is contained into $W$ for every $s>\frac 12$, $u_N(t)$ tends to $u(t)$ in $W$ uniformly for $t\in [0,T]$. By Proposition \ref{prop}, there exists a constant $C_1>0$ such that
\begin{align} \label{C1}
\norm{u(t)}_{W}+1\leq C_1, \text{\;\;for all\;\;} t\in \reals.
\end{align}
Consequently, there exists $N'\in \mathbb N$ such that
\begin{align} \label{ge-0}
\norm{u_N(t)}_{W}\leq \norm{u(t)}_{W}+1\leq C_1,
\text{\;\;for all\;\;} N>N', \;\;t\in [0,T].
\end{align}
Also, recall that the initial condition $u_0 \in G_{\sigma}(W)$, i.e., $\norm{u_0}_{G_{\sigma}(W)}<\infty$. Since $u_N(0)=P_N u_0$, one has
\begin{align} \label{C2}
\sum_{k=0}^N e^{\sigma k}|\hat u_{N}(0,k)|\leq \norm{u_0}_{G_{\sigma}(W)}.
\end{align}
Define
\begin{align} \label{C0}
C_0:=\max\left\{\norm{u_0}_{G_{\sigma}(W)},\frac{1+\sqrt{5}}{2}e C_1 \right\},
\end{align}
where $C_1$ has been specified in (\ref{C1}).

Let us fix an arbitrary $N>N'$. We aim to prove
\begin{align} \label{cla}
\sum_{k=0}^N e^{\t(t)k}|\hat u_{N}(t,k)|\leq C_0, \text{\;\;for all\;\;} t\in [0,T],
\end{align}
with $\t(t)>0$ that will be specified in (\ref{def-tau}), below.

Notice, due to (\ref{Appro-Sol}) and (\ref{Galerkin}), we infer
\begin{align*}
\frac{d}{dt}\hat u_N(t,k)=-i\sum_{n-j+m=k \atop 0\leq n,j,m\leq  N}\hat u_N(t,n) \overline {\hat u_N(t,j)} \hat u_N(t,m), \;\;t\in [0,T], \;k=0,1,\ldots,N.
\end{align*}
Then, one can easily find that
\begin{align} \label{z11}
\frac{d}{dt} |\hat u_N(t,k)|
\leq \sum_{n-j+m=k \atop 0\leq n,j,m\leq N}|\hat u_N(t,n)| |\hat u_N(t,j)| |\hat u_N(t,m)| ,
\end{align}
for $k=0,1,\ldots,N$, and all $t\in [0,T]$.

In order to estimate the Gevrey norm, we consider
\begin{align*}
&\frac{d}{dt}\left(e^{\t(t)k}|\hat u_N(t,k)|\right)\notag\\
&=\t'(t)k e^{\t(t) k}|\hat u_N(t,k)|+e^{\t(t) k}\frac{d}{dt}|\hat u_N(t,k)|\notag\\
&\leq \t'(t)k e^{\t(t) k}|\hat u_N(t,k)|+e^{\t(t) k}\sum_{n-j+m=k \atop 0\leq n,j,m\leq N}|\hat u_N(t,n)| |\hat u_N(t,j)| |\hat u_N(t,m)|,
\end{align*}
for $k=0,1,\ldots,N$, and $t\in [0,T]$, where (\ref{z11}) has been used in the last inequality.

Summing over all integers $k=0,1,\cdots,N$ yields
\begin{align} \label{z1}
&\frac{d}{d t}\left(\sum_{k=0}^Ne^{\t(t) k}|\hat u_N(t,k)|\right) \notag\\
&\leq \t'\sum_{k=0}^N k e^{\t k }|\hat u_N(k)|+
\sum_{k=0}^N  e^{\t k} \left(\sum_{n-j+m=k \atop 0\leq n,j,m\leq N}|\hat u_N(n)| |\hat u_N(j)| |\hat u_N(m)|\right) \notag\\
&=\t'\sum_{k=0}^N k e^{\t k }|\hat u_N(k)|+
\sum_{k=0}^N \left(\sum_{n-j+m=k \atop 0\leq n,j,m\leq N} e^{\t n}|\hat u_N(n)| e^{-\t j} |\hat u_N(j)| e^{\t m}|\hat u_N(m)|\right)\notag\\
&\leq \t'\sum_{k=0}^N k e^{\t k }|\hat u_N(k)|+
\left(\sum_{k=0}^Ne^{\t k}|\hat u_N(k)|\right)^2 \left(\sum_{k=0}^N |\hat u_N(k)|\right),
\end{align}
where the last formula is obtained by using the Young's convolution inequality
and the fact $e^{-\t j}\leq 1$, for $\t$, $j\geq 0$.

Now, we estimate the second term on the right-hand side of (\ref{z1}). The key ingredient of the calculation is the elementary inequality $e^x\leq e+x^{\ell}e^x$, for all $x\geq 0$, $\ell\geq 0$, and we select $\ell=\frac{1}{2}$ here. Hence
\begin{align} \label{z2}
&\left(\sum_{k=0}^Ne^{\t k}|\hat u_N(k)|\right)^2 \left(\sum_{k=0}^N |\hat u_N(k)|\right) \notag\\
&\leq \left(\sum_{k=0}^N e|\hat u_N(k)|
+\sum_{k=0}^N\t^{\frac{1}{2}} k^{\frac{1}{2}}e^{\t k}|\hat u_N(k)|\right)^2
\left(\sum_{k=0}^N |\hat u_N(k)|\right) \notag\\
&\leq 2 e^2 \left(\sum_{k=0}^N |\hat u_N(k)|\right)^3
+2 \t \left(\sum_{k=0}^N k e^{\t k}|\hat u_N(k)|\right)
\left(\sum_{k=0}^N e^{\t k}|\hat u_N(k)|\right)\left(\sum_{k=0}^N |\hat u_N(k)|\right),
\end{align}
where we have used Young's inequality and H\"older's inequality.

Thus, combining (\ref{z1}) and (\ref{z2}) yields
\begin{align} \label{ge-1}
&\frac{d}{dt}\left(\sum_{k=0}^N e^{\t(t)k}|\hat u_N(t,k)|\right) \notag \\
&\leq \t'(t)\sum_{k=0}^N k e^{\t(t) k}|\hat u_N(t,k)|+2 e^2 \left(\sum_{k=0}^N |\hat u_N(t,k)|\right)^3 \notag\\
&\hspace{0.2 in}+2 \t(t) \left(\sum_{k=0}^N k e^{\t(t) k} |\hat u_N(t,k)|\right) \left(\sum_{k=0}^N e^{\t(t) k} |\hat u_N(t,k)|\right) \left(\sum_{k=0}^N |\hat u_N(t,k)|\right) \notag\\
&\leq \frac{1}{2}\t'(t)\sum_{k=0}^N k e^{\t(t) k}|\hat u_N(t,k)|+ 2 e^2 C_1^3 \notag\\
&\hspace{0.2 in}+\left(\frac{1}{2}\t'(t)+2 C_1 \t(t) \sum_{k=0}^N e^{\t(t) k} |\hat u_N(t,k)| \right)\left(\sum_{k=0}^N k e^{\t(t) k} |\hat u_N(t,k)|\right),
\end{align}
for all $t\in [0,T]$, where we have used (\ref{ge-0}).

Denote by $\t_N(t)$, $t\in [0,t_N]$, the unique solution of the ODE
\begin{align} \label{z-N}
\frac{1}{2}\t_N'(t)+2 C_1 \t_N(t)z_N(t)=0, \text{\;\;with\;\;} \t_N(0)=\sigma,
\end{align}
where we set
\begin{align} \label{def-z}
z_N(t):=\sum_{k=0}^N e^{\t_N(t)k}|\hat u_N(t,k)|.
\end{align}
Due to (\ref{z-N}) and (\ref{def-z}), we infer from (\ref{ge-1}) that
\begin{align} \label{ge-1'}
\frac{d z_N}{dt}(t) &\leq \frac{1}{2}\t_N'(t)\sum_{k=0}^N k e^{\t_N(t) k}|\hat u_N(t,k)|+2 e^2 C_1^3 \notag\\
&\leq -2 C_1 z_N(t)\t_N(t)\sum_{k=0}^N k e^{\t_N(t) k}|\hat u_N(t,k)|+2 e^2 C_1^3, \;\;t\in [0,t_N].
\end{align}
Next, we estimate $\t_N(t)\sum_{k=0}^N k e^{\t_N(t) k}|\hat u_N(t,k)|$ by considering the following two cases:

\emph{Case 1}: $N\geq \frac{1}{\t_N(t)}$. In this case, one has
\begin{align}   \label{ge-2}
&\t_N(t)\sum_{k=0}^N k e^{\t_N(t) k}|\hat u_N(t,k)| \geq \t_N(t)\sum_{\frac{1}{\t_N(t)} \leq k \leq N} k e^{\t_N(t) k}|\hat u_N(t,k)|\notag\\
&\geq \sum_{\frac{1}{\t_N(t)} \leq k \leq N} e^{\t_N(t) k}|\hat u_N(t,k)|
=\sum_{k=0}^N e^{\t_N(t) k}|\hat u_N(t,k)|-\sum_{0\leq k<\frac{1}{\t_N(t)}} e^{\t_N(t) k}|\hat u_N(t,k)| \notag\\
&\geq z_N(t)-e \sum_{0\leq k<\frac{1}{\t_N(t)}} |\hat u_N(t,k)|\geq z_N(t)-e C_1,
\end{align}
where the fact (\ref{ge-0}) has been used.

\emph{Case 2}: $N<\frac{1}{\t_N(t)}$. In this case, in order to obtain the same estimate as (\ref{ge-2}), we proceed as follows:
\begin{align*}
& \t_N(t)\sum_{k=0}^N k e^{\t_N(t) k}|\hat u_N(t,k)| \geq 0=z_N(t)-
\sum_{k=0}^N e^{\t_N(t) k}|\hat u_N(t,k)| \notag\\
& \geq z_N(t)-
e \sum_{k=0}^N |\hat u_N(t,k)|\geq z_N(t)-e C_1.
\end{align*}

We conclude from the above two cases that
$$\t_N(t)\sum_{k=0}^N k e^{\t_N(t) k}|\hat u_N(t,k)| \geq z_N(t)-e C_1,$$
and by substituting it into (\ref{ge-1'}), one has
\begin{align} \label{ge-3}
\frac{d z_N}{dt}(t)\leq -2 C_1 z_N^2(t)+2 e C_1^2 z_N(t)+2 e^2 C_1^3, \text{\;\;for all\;\;} t\in [0,t_N].
\end{align}
Notice that the right-hand side of (\ref{ge-3}) is negative when  $z_N > z^*=\frac{1+\sqrt{5}}{2}e C_1$, and hence (\ref{ge-3}) implies that
\begin{align} \label{bound-z}
z_N(t)\leq \max\{z_N(0),z^*\}=\max\left\{\sum_{k=0}^N e^{\sigma k}|\hat u_N(0,k)|,\frac{1+\sqrt{5}}{2}e C_1\right\}\leq C_0,
\end{align}
for all $t\in [0,t_N]$, where we have also used (\ref{C2}) and (\ref{C0}) in the above estimate. Therefore, by virtue of the uniform bound (\ref{bound-z}) of $z_N(t)$, the solution $\t_N(t)$ of the initial value problem (\ref{z-N}) on $[0,t_N]$ can be extended to the solution on $[0,T]$, and thus (\ref{bound-z}) holds for all $t\in [0,T]$, i.e.,
\begin{align} \label{bound-z'}
z_N(t)\leq C_0, \text{\;\;for all\;\;} t\in [0,T],
\end{align}
and along with (\ref{z-N}), we infer
\begin{align} \label{ge-4}
\t_N(t)=\sigma \exp\left(-4 C_1 \int_0^t z_N(s) ds\right)\geq \sigma e^{-4C_0 C_1 t}, \text{\;\;for all\;\;} t\in [0,T].
\end{align}
Let us define
\begin{align} \label{def-tau}
\t(t)=\sigma e^{-\l |t|}, \text{\;\;with\;\;} \l=4 C_0 C_1,
\end{align}
where $C_0$ and $C_1$ are specified in (\ref{C0}) and (\ref{C1}), respectively.
Then, (\ref{ge-4}) and (\ref{def-tau}) show that $\t(t)\leq \t_N(t)$ on $[0,T]$, and consequently,
\begin{align} \label{ge-5}
\norm{u_N(t)}_{G_{\t(t)}(W)}=\sum_{k=0}^N e^{\t(t)k}|\hat u_N(t,k)|
\leq \sum_{k=0}^N e^{\t_N(t)k}|\hat u_N(t,k)|=z_N(t)\leq C_0,
\end{align}
for all $t\in [0,T]$, due to (\ref{bound-z'}). Since $N$ is an arbitrary integer larger than $N'$, we conclude, for every fixed number $N_0$,
for every $t\in [0,T]$,
$$\sum_{k=0}^{N_0} e^{\t(t)k}|\hat u(t,k)|=\lim _{N\rightarrow \infty }\sum_{k=0}^{N_0} e^{\t(t)k}|\hat u_N(t,k)|\le C_0\ .$$
Therefore, since $N_0\ge 0$ and $T>0$ are arbitrarily selected, $\norm{u(t)}_{G_{\t(t)}(W)}\leq C_0$ for all $t\geq 0$.
\end{proof}

\smallskip
\begin{remark}\label {improvement} In Theorem \ref{main-thm}, we found a lower bound $\t(t)$
of the radius of spatial analyticity of $u(t)$, where $\t(t)=\sigma e^{-\l |t|}$, with $\l=4C_0 C_1$. By the definition of $C_0$ in (\ref{C0}), one has
\begin{align} \label{def-lamda}
\l=
\begin{cases}
2(1+\sqrt{5})e C_1^2, \text{\;\;if\;\;} \norm{u_0}_{G_{\sigma}(W)}\leq (1+\sqrt 5)e C_1/2;\\
4 C_1 \norm{u_0}_{G_{\sigma}(W)}, \text{\;\;if\;\;} \norm{u_0}_{G_{\sigma}(W)}> (1+\sqrt 5)e C_1/2.
\end{cases}
\end{align}
Here, we shall provide a slightly different lower bound $\tilde \t(t)$ of the radius of analyticity of $u(t)$. More precisely, we can choose $\tilde \t(t)=\sigma e^{-\tilde \l(t) |t|}$, where $\tilde \l(t)$ defined in (\ref{re-4}) below, is almost \emph{independent} of the Gevrey norm $\norm{u_0}_{G_{\sigma}(W)}$ of the initial datum, for large values of $|t|$. Indeed, by (\ref{ge-3}), it is easy to see that
\begin{align} \label{re-1}
\frac{dz_N}{dt}(t)\leq -2 C_1 \left(z_N(t)-\frac{eC_1}{2}\right)^2+\frac{5}{2}e^2C_1^3.
\end{align}
After some manipulations of (\ref{re-1}), we obtain
\begin{align} \label{re-2}
\int_0^t \left(z_N(s)-\frac{eC_1}{2}\right)^2 ds \leq \frac{z_N(0)}{2C_1}+\frac{5e^2 C_1^2 t}{4} \leq \frac{\norm{u_0}_{G_{\s}(W)}}{2C_1}+\frac{5e^2 C_1^2 t}{4}.
\end{align}
Note
\begin{align} \label{re-3}
\int_0^t z_N(s) ds&=\int_0^t \left(z_N(s)-\frac{e C_1}{2}\right) ds+\frac{e C_1}{2}t \notag\\
&\leq \left[\int_0^t \left(z_N(s)-\frac{e C_1}{2}\right)^2 ds\right]^{\frac{1}{2}}\sqrt{t}+\frac{e C_1}{2}t \notag\\
&\leq \left[\frac{\norm{u_0}_{G_{\s}(W)}}{2C_1}+\frac{5e^2 C_1^2 t}{4}\right]^{\frac{1}{2}}\sqrt{t}+\frac{e C_1}{2}t,
\end{align}
where we have used the estimate (\ref{re-2}). Thus, by (\ref{ge-4}) and (\ref{re-3}), we may select
\begin{align} \label{re-4}
\tilde \tau(t)=\s e^{-\tilde \l(t)|t|}, \text{\;\;with\;\;} \tilde \l(t)=2C_1\left[\frac{2\norm{u_0}_{G_{\s}(W)}}{C_1|t|}+5e^2 C_1^2\right]^{\frac{1}{2}}+2e C_1^2, \;\;|t|>0,
\end{align}
and then $\tilde \tau(t)\leq \tau_N(t)$. Thus, by adopting the argument in Theorem \ref{main-thm}, it can be shown that $\norm{u(t)}_{G_{\tilde \t(t)}(W)}\leq C_0$ for all $t\in \reals$. Also, we see from (\ref{re-4}) that $\tilde \l(t)\rightarrow 2(1+\sqrt 5)e C_1^2$ as $|t|\rightarrow \infty$, that is, $\tilde \l(t)$ is almost \emph{independent} of $\norm{u_0}_{G_{\sigma}(W)}$, for large values of $|t|$, in contrast to the definition (\ref{def-lamda}) of $\l$.

\end{remark}

\smallskip

\begin{remark} For analytic initial data, the  Gevrey norm estimate $\norm{u(t)}_{G_{\t(t)}(W)}\leq C_0$, where $\t(t)=\sigma e^{-\l |t|}$, can provide a growth estimate of the $H^s$ norm of the solution $u(t)$.
Indeed,
\begin{align*}
\norm{u}_{H^s}^2=\sum_{k\geq 0} (k^{2s}+1)|u_k|^2
\leq \sup |u_k| \left(\sum_{k\geq 0} |u_k| e^{\t k} \frac{k^{2s}}{e^{\t k}}+\sum_{k\geq 0}|u_k|\right).
\end{align*}
Since the maximum of the function $k \mapsto \frac{k^{2s}}{e^{\t k}}$ occurs at $k=\frac{2s}{\t}$, we obtain
\begin{align*}
\norm{u}_{H^s}^2\leq \norm{u}_{W}\left[ e^{-2s} \left(\frac{2s}{\t}\right)^{2s}
\norm{u}_{G_{\t}(W)}+\norm{u}_{W}\right].
\end{align*}
It follows that
 \[
\norm{u(t)}_{H^s}^2\leq C(s)  e^{2\l s t},
\]
that is to say, the $H^s$ norm grows at most exponentially, if $s>\frac{1}{2}$, which agrees with the $H^s$ norm estimates in Corollary 2, section 3 of \cite{Gerard-10}.
\end{remark}

\smallskip

\begin{remark} \label{rmk-compare}
Let us set $A=\sqrt{I-\Delta}$. Recall the regular Gevrey classes of analytic functions are defined by $\mathcal D(A^s e^{\s A})$ furnished the norm $\norm{A^s e^{\s A}\cdot}_{L^2(\mathbb T)}$, where $s\geq 0$, $\s>0$. It has been mentioned in the Introduction that we choose to employ the special Gevrey class $G_{\s}(W)$ in this manuscript, since it provides better estimate of the lower bound the radius of analyticity of the solution. In particular, we can do the following comparisons.

Suppose the initial condition $u_0\in \mathcal D(A^s e^{\sigma A})$, $s>\frac{1}{2}$, $\s>0$, and let us perform the estimates by using the regular Gevrey classes $\mathcal D(A^s e^{\sigma A})$. Adopting similar arguments as in \cite{Guo-Titi-12, Larios-Titi-10}, one can manage to show that
\begin{align*}
\norm{A^s e^{\t_1(t) A}u(t)}_{L^2}^2\leq  \norm{A^s e^{\sigma A}u_0}_{L^2}^2+C\int_0^{|t|} \norm{u(t')}_{H^s}^4 dt',
\;\;s>\frac{1}{2},
\end{align*}
if $\t_1(t)=\sigma e^{-\int_0^{|t|} h(t') dt'}$,
where $h(t)=C\left(\norm{A^p e^{\sigma A}u_0}_{L^2}^2+\int_0^{|t|} \norm{u(t')}_{H^s}^4 dt'\right)$. Since $\norm{u(t)}_{H^s}$, $s>\frac{1}{2}$, has an upper bound that grows exponentially as $|t|\rightarrow \infty$ (see \cite{Gerard-10}), we infer that $\t_1(t)$ might shrinks \emph{double} exponentially, compared to the exponential shrinking rate of $\t(t)$ established in Theorem \ref{main-thm}, where the Gevrey class $G_{\s}(W)$ is used. Such advantage of employing the special Gevrey class $G_{\s}(W)$ stems from the uniform boundedness of the norm $\norm{u(t)}_{W}$ for the solution $u$ to the cubic Szeg\H{o} equation for sufficiently regular initial data.
\end{remark}

\section{Appendix}
For the sake of completion, we provide a straightforward proof of the following property of the Hankel operator.
\begin{proposition}
For any $u\in L^2_+(\mathbb T)\cap W$, the following double inequality holds
\begin{align} \label{tracebound1}
\frac 12\norm{u}_{W}\leq  Tr(|H_u|)\le \sum _{k=0}^\infty \left (\sum _{\ell =0}^\infty \vert \hat u(k+\ell )\vert ^2\right )^{\frac12}\ .
\end{align}
\end{proposition}
\begin{proof}
Recall the following result in the operator theory (see, e.g.,  \cite{Conway-2000}).
Let $A$ be an operator on a Hilbert space $H$, where $A$ belongs to the trace class. If $\{e_k\}$ and $\{f_k\}$ are two orthonormal families in $H$, then
\begin{align} \label{operator}
\sum_{k}|(A e_k,f_k)|\leq Tr(|A|).
\end{align}
In order to find a lower bound of $Tr(|H_u|)$, we use the estimate (\ref{operator}) by computing
$\sum_{k}|( H_u (e^{ik\theta}),f_k)|$ with two different orthonormal systems $\{f_k\}$ selected below.
Notice that, by the definition (\ref{hankel}) of the Hankel operator $H_u: L^2_+(\mathbb T)\rightarrow L^2_+(\mathbb T)$, we have
\begin{align} \label{appen-0}
H_u (e^{ik\theta})=\Pi(u e^{-ik\theta})=\Pi\left(\sum_{j\geq 0}\hat u(j) e^{i(j-k)\theta}\right)
=\sum_{j\geq 0}\hat u(j+k)e^{ij\theta}.
\end{align}
If we choose $f_k=e^{ik\theta}$, $k\geq 0$, and use (\ref{appen-0}), then it follows that
\begin{align*}
Tr (\vert H_u\vert )\ge \sum_{k\geq 0}|(H_u (e^{ik\theta}),e^{ik\theta})|
=\sum_{k\geq 0}\left|\big(\sum_{j\geq 0}\hat u(j+k) e^{ij\theta},e^{ik\theta}\big)\right|=\sum_{k\geq 0}|\hat u(2k)|.
\end{align*}
However, if we select $f_k=e^{i(k+1)\theta}$, for every integer $k\geq 0$, then
\begin{align*}
Tr (\vert H_u\vert )\ge \sum_{k\geq 0}|(H_u (e^{ik\theta}),f_k)|
=\sum_{k\geq 0 } |\hat u(2k+1)|.
\end{align*}
Summing up, we have proved
\begin{align*}
2 Tr(|H_u|)\geq \sum_{k\geq 0}|\hat u(k)|=\norm{u}_{W}.
\end{align*}

We now pass to the second inequality. Recall from (\ref{hankel}) that, for every $h_1,h_2\in L^2_+$,
$$(H_u(h_1), h_2)=(u, h_1h_2)=(H_u(h_2), h_1)\ ,$$
which implies that $H_u^2$ is a positive self-adjoint linear operator. Moreover,
$$Tr(H_u^2)=\sum _{k,\ell \ge 0}\vert \hat u(k+\ell )\vert ^2=\sum _{n=0}^\infty (n+1)\vert \hat u(n)\vert ^2<\infty $$
as soon as $u\in H^{1/2}$.  In other words, $\vert H_u\vert =\sqrt {H_u^2}$ is a positive Hilbert--Schmidt operator if $u\in L^2_+\cap H^{1/2}.$
Let $\{ \rho _j\}$ be the sequence of positive eigenvalues of $\vert H_u\vert $,
and let $\{ \e _j\}$ be an orthonormal sequence of corresponding eigenvectors. Notice that
$$(H_u(\e _j), H_u(\e _{j'}))=(H_u^2(\e _{j'}), \e _j)=\rho _{j'}^2\delta _{jj'}\ .$$
We infer that the sequence $\{ H_u(\e _j)/\rho _j\} $ is orthonormal. We then define the following antilinear operator on $L^2_+$,
$$\Omega _u(h)=\sum _j \frac{(H_u(\e _j), h)}{\rho _j}\e _j\ .$$
Notice that, due to the orthonormality of both systems $\{ \e _j\} $ and $\{ H_u(\e _j)/\rho _j\}$,
$$\norm {\Omega _u(h)} \le \norm{h}\ .$$
We now observe that
\begin{eqnarray*}
\rho _j&=&(\Omega _u(H_u(\e _j)), \e _j)=\sum _{k=0}^\infty (\Omega _u(e^{ik\theta}) , \e _j) (e^{ik\theta },  H_u(\e _j))=\sum _{k=0}^\infty (\Omega _u(e^{ik\theta}), \e _j)(\e _j, H_u(e^{ik\theta }))\\
&=&\sum _{k,\ell \ge 0}\overline {\hat u(k+\ell )}(\Omega _u(e^{ik\theta }), \e _j)(\e _j, e^{i\ell \theta})\ ,
\end{eqnarray*}
and therefore, for every $N$,
$$Tr(\vert H_u\vert )=\sum _j \rho _j=\sum _{k,\ell \ge 0}\overline {\hat u(k+\ell )}(\Omega _u(e^{ik\theta}), e^{i\ell \theta})\ .$$
Apply the Cauchy--Schwarz inequality to the sum on $\ell $,
$$Tr(\vert H_u\vert )\le \sum _{k=0}^\infty \Vert \Omega _u(e^{ik\theta})\Vert \left (\sum _{\ell =0}^\infty \vert \hat u(k+\ell )\vert ^2\right )^{\frac12}\ ,$$
and the claim follows from $\norm{\Omega _u(e^{ik\theta })}\le \norm{e^{i\theta}}= 1$.
\end{proof}
Using the above proposition, it is easy to derive the estimate (\ref{tracebound}) used in the proof of Proposition \ref{prop}. Indeed, by the Cauchy--Schwarz inequality in the $k$ sum, we have, for every $s>1$,
\begin{eqnarray*}
\sum _{k=0}^\infty \left (\sum _{\ell =0}^\infty \vert \hat u(k+\ell )\vert ^2\right )^{\frac12}&\le & \left (\sum _{k=0}^\infty (1+k)^{1-2s}\right )^{\frac 12}
 \left (\sum _{k,\ell \ge 0}(1+k)^{2s-1}\vert \hat u(k+\ell )\vert ^2\right )^{\frac 12} \\
&\le &\left (\frac{s}{s-1}\right )^{\frac 12} \left (\sum _{k,\ell \ge 0}(1+k+\ell )^{2s-1}\vert \hat u(k+\ell )\vert ^2\right )^{\frac 12}\\
&\le & C_s\norm{u}_{H^s}\ .
\end{eqnarray*}

\par\smallskip\noindent
\textbf{Acknowledgement\,:}    This work was  supported in part by the Minerva Stiftung/Foundation, and  by the NSF
grants DMS-1009950, DMS-1109640 and DMS-1109645.


\begin{thebibliography}{99}
\bibitem{Burq-Gerard-Tzvetkov-02} N. Burq, P. G\'erard, N. Tzvetkov, An instability property of the nonlinear Schr\"odinger equation on $S^d$. Math. Res. Lett. 9 (2002), no. 2-3, 323-335.

\bibitem{Burq-Gerard-Tzvetkov-05} N. Burq, P. G\'erard, N. Tzvetkov, Bilinear eigenfunction estimates and the nonlinear Schr\"odinger equation on surfaces. Invent. Math. 159 (2005), no. 1, 187-223.

\bibitem{Conway-2000} J. B. Conway, A course in operator theory, Graduate Studies in Mathematics 21, American Mathematical Society, Providence, RI, 2000.

\bibitem{Ferrari-Titi-98} A. B. Ferrari, E. S. Titi, Gevrey regularity for nonlinear analytic parabolic equations,
Communications in Partial Differential Equations 23 (1998), 1-16.

\bibitem{Foias-Temam-89} C. Foias, R. Temam, Gevrey class regularity for the solutions of the Navier-Stokes equations, Journal of Functional Analysis 87 (1989), 359-369.

\bibitem{Gerard-10} P. G\'erard, S. Grellier, The cubic Szeg\H{o} equation, Annales Scientifiques de l'Ecole Normale Sup\'erieure 43 (2010), 761-810.

\bibitem{Gerard-12} P. G\'erard, S. Grellier, Invariant tori for the cubic Szeg\H{o} equation, Invent. Math. 187 (2012), no. 3, 707-754.

\bibitem{Gerard-13} P. G\'erard, S. Grellier, An explicit formula for the cubic Szeg\H{o} equation, preprint (2013),  arXiv:1304.2619.

\bibitem{Guo-Titi-12} Yanqiu Guo, Edriss S. Titi, Persistency of analyticity for quasi-linear wave equations: an energy-like approach, preprint (2013), arXiv:1301.0137.

\bibitem{Kukavica-Vicol-09} I. Kukavica, V. Vicol, On the radius of analyticity of solutions to the three-dimensional Euler equations, Proceedings of the American Mathematical Society 137 (2009), 669-677.

\bibitem{Kuksin} S. B. Kuksin, Oscillations in space--periodic nonlinear Schr\"odinger equations, Geom. Funct. Anal. 2 (1997), 338--363.

\bibitem{Larios-Titi-10} A. Larios, E. S. Titi, On the higher-order global regularity of the inviscid Voigt-regularization of three-dimensional
hydrodynamic models, Discrete and Continuous Dynamical Systems Series B 14 (2010), 603-627.

\bibitem{Levermore-Oliver-97} C. D. Levermore, M. Oliver, Analyticity of solutions for a generalized Euler equation, Journal of Differential Equations 133 (1997), 321-339.

\bibitem{Oh-11} T. Oh, Remarks on nonlinear smoothing under randomization for the periodic KdV and the cubic Szeg\H{o} equation, Funkcial. Ekvac. 54 (2011), no. 3, 335-365.

\bibitem{Oliver-Titi-00} M. Oliver, E. S. Titi, Remark on the rate of decay of higher order derivatives for solutions to the Navier-Stokes equations in $\reals^n$, Journal of Functional Analysis 172 (2000), 1-18.

\bibitem{Oliver-Titi-01} M. Oliver, E. S. Titi, On the domain of analyticity for solutions of second order analytic nonlinear differential equations,
Journal of Differential Equations 174 (2001), 55-74.

\bibitem{Pe80} V. V. Peller, Hankel operators of class
$\mathfrak{S}_p$ and their applications (rational approximation,
Gaussian processes, the problem of majorization of operators), Math.
USSR Sb. 41 (1982), 443-479.

\bibitem{P} V. V. Peller, Hankel operators and their applications, Springer Monographs in Mathematics, Springer-Verlag, New York, 2003.

\bibitem{Pocovnicu-11} O. Pocovnicu, Explicit formula for the solution of the Szeg\H{o} equation on the real line and applications, Discrete Contin. Dyn. Syst. 31 (2011), no.3, 607-649.

\bibitem{Pocovnicu-11'} O. Pocovnicu, Traveling waves for the cubic Szeg\H{o} equation on the real line, Anal. PDE 4 (2011), no. 3, 379-404.

\bibitem{Xu-13} H. Xu, Large time blow up for a perturbation of the cubic Szeg\H{o} equation, paper in preparation.



\end{thebibliography}
\end{document}